\documentclass[reqno,12pt, a4paper]{amsart}
\usepackage{enumerate, longtable}
\usepackage{amsmath, amscd, amsfonts, amsthm, amssymb, latexsym, comment, stmaryrd, graphicx}
\usepackage{xcolor}
\usepackage[all]{xy}
\usepackage{a4wide}
\usepackage{nicefrac}
\usepackage{calrsfs}
\usepackage{mathtools}

\definecolor{myblue}{rgb}{0.09,0.32,0.44} %22-84-113
\usepackage[
  unicode=true,
  pdfusetitle,
  bookmarks=true,
  bookmarksnumbered=false,
  bookmarksopen=false,
  breaklinks=false,
  pdfborder={0 0 0},
  pdfborderstyle={},
  backref=false,
  colorlinks=true,
  linkcolor=myblue,
  citecolor=myblue,
  urlcolor=blue
]{hyperref}

\theoremstyle{plain}
\newtheorem{theorem}{Theorem}%[section]

\newtheorem{lemma}[theorem]{Lemma}

\newtheorem*{conjecture*}{Conjecture}

\newtheorem{claim}{Claim}

\theoremstyle{definition}

\newcommand{\Gal}{{\rm Gal}}

\newcommand{\eps}{\varepsilon}
\newcommand{\Y}{\Psi}

\newcommand{\red}{{\rm {red}}}

\newcommand{\ZZ}{\mathbb{Z}}
\newcommand{\NN}{\mathbb{N}}
\newcommand{\PP}{\mathbb{P}}

\newcommand{\CC}{\mathbb{C}}
\newcommand{\FF}{\mathbb{F}}

\def\clap#1{\hbox to 0pt{\hss#1\hss}}

\def\FF{\mathbb{F}}
\def\QQ{\mathbb{Q}}

\def\mB{\mathcal{B}}

\def\mQ{\mathcal{Q}}
\def\mR{\mathcal{R}}

\def\disc{\Delta} %discriminnat

\begin{document}

\title{Irreducible polynomials of bounded height}
\author{Lior Bary-Soroker}
\address{Raymond and Beverly Sackler School of Mathematical Sciences, Tel Aviv University, Tel Aviv 69978, Israel}
\email{barylior@post.tau.ac.il}
\author{Gady Kozma}
\address{Department of Mathematics, The Weizmann Institute of Science, Rehovot 76100, Israel}
\email{gady.kozma@weizmann.ac.il}
%\keywords{totally $S$-adic numbers, Hilbertian fields, irreducible polynomials}
%\subjclass[2010]{12E25, 12E30, 12F10}

\begin{abstract}
The goal of this paper is to prove that a random polynomial with i.i.d.\ random coefficients taking values uniformly in $\{1,\ldots, 210\}$ is irreducible with probability tending to $1$ as the degree tends to infinity. Moreover, we prove that the Galois group of the random polynomial contains the alternating group, again with probability tending to $1$.
\end{abstract}

\maketitle

\section{Introduction}
The study of random polynomials has a long history. One direction, that we do no pursue here, is the study of the distribution of the roots of the polynomial. Notable phenomena about the roots are the logarithmic number of real roots, \cite{BP,Kac,LO} and the fact that they cluster near the unit circle with asymptotically uniform distribution \cite{ET50, SV95, IZ97}. Additional surprising phenomena have been observed in extensive simulations of the roots of random polynomials with coefficients $\pm 1$, such as the appearance of various approximate Julia sets \cite{B11}.
%The classical results on the number of real roots of a random polynomials are given by Bloch-P\'olya \cite{BP}, Littlewood-Offord \cite{LO}, and Kac \cite{Kac} (in different randomality models).

When the coefficients of the random polynomial 
\[
f(X) = X^n + \sum_{i=0}^{n-1} \zeta_i X^i
\] 
are integral; i.e., $\zeta_0,\ldots, \zeta_{n-1}\in \ZZ$, 
a natural question is about irreducibility over $\QQ$; that is to say, as an element of the ring $\QQ[X]$. 
A more subtle question is about the distribution of the Galois group $G_f$ of the random polynomial $f$, which is by definition the Galois group of the splitting field of $f$ over $\QQ$. This group $G_f$ may be considered as a subgroup of the symmetric group $S_n$ via the action on the roots of $f$. A basic property of Galois theory that connects irreducibility and Galois groups, is that 
\begin{quote}
\it $f$ is irreducible if and only if $G_f$ is transitive. 
\end{quote}
It is believed, and proved in many cases, that with high probability $f$ is irreducible, and in fact `the most irreducible' in the sense that its Galois group is the full symmetric group. 

In the simplest model, \emph{the large box model}, one fixes $n=\deg f$ and the coefficients $\zeta_0,\ldots, \zeta_{n-1}$ are i.i.d.\ random variables taking values uniformly in $\{-L,\ldots, L\}$ with $L\to \infty$. We do not know how far back it goes, but it is well known that
\begin{equation}\label{bbm_irre}
\lim_{L\to \infty} \PP(f \mbox{ is irreducible}) = 1.
\end{equation}
The state-of-the-art error term in \eqref{bbm_irre} is given by Kuba \cite{Kuba}: $\PP(f \mbox{ is reducible}) = O(L^{-1})$,  $n>2$. In the more subtle case of Galois groups, the first result goes  back at least to 
Van der Waerden \cite{VanDerWaerden} who proved that 
\begin{equation}\label{bbm_Gal}
\lim_{L\to \infty} \PP(G_f = S_n) = 1.
\end{equation}
Van der Waerden's error term in \eqref{bbm_Gal} is explicit: $\PP(G_f\neq S_n)=O(L^{-\frac{1}{6}})$. It was improved in \cite{Dietmann, Gallagher} using sieve methods, and the state-of-the-art result was given recently by Rivin \cite{Rivin} using an elementary method: $\PP(G_f\neq S_n)=O(L^{-1+\epsilon})$. This is nearly optimal, since 
\[
\PP(G_f\neq S_n) \geq \PP(G_f\leq S_{n-1}) \geq \PP(f(0)=0) = \frac{1}{2L+1}.
\]

Another model, which is the focus of investigation of this paper, is \emph{the restricted coefficients model}. In this model, the coefficients of $f$ are i.i.d.\ random variables taking values uniformly in a \emph{fixed} finite set, and the degree $n=\deg f$ grows to infinity. Two well studied examples are $\pm 1$ coefficients \cite[and references within]{someguy} and $0,1$ coefficients \cite{OP93} (with $\zeta_0=1$). 

In the latter model, Konyagin \cite{K99} proves that 
\begin{equation}
\lim_{n\to \infty} \PP( f \mbox{ has all irreducible factors of degree $\geq c n/\log n$}) =1.
\end{equation}
This implies that \cite[page 334]{K99} 
\[
\PP(f \mbox{ is irreducible}) \geq \frac{c}{\log n}.
\]
This is the state-of-the-art result on $\zeta_i\in \{0,1\}$, although it merely says that the probability does not tend to $0$ too rapidly, while the truth is that it tends to 1.

The restricted coefficient model is considered much more difficult than the large box model, mainly because the methods of the large box model are not applicable as they are based on reductions modulo large primes. 

Our main result seems to be the first establishment of the analogue of \eqref{bbm_irre} in the restricted coefficients model:

\begin{theorem}\label{thm:irred}
Let $L$ be a positive integer divisible by at least $4$ distinct primes. Let 
\[
f= X^n+\sum_{i=0}^{n-1} \zeta_i X^i
\] 
be a polynomial, where $\zeta_1,\zeta_2,\ldots$ are i.i.d.\ random variables taking values uniformly in $\{1,\ldots, L\}$. Then 
\[
\lim_{n\to \infty}\PP(f \mbox{ is irreducible}) =1.
\]
\end{theorem}

Note that the smallest $L$ that satisfies the restriction of the theorem is 
\[
L=210=2\times 3\times 5\times 7.
\]  

Under the same conditions as in Theorem~\ref{thm:irred} we can also show that the Galois group of $f$ is either $S_n$ (the whole symmetric group) or $A_n$ (the alternating group, the group of even permutations). We find it worthwhile to note that the irreducibility of $f$ (or, in other words, transitivity of the Galois group) is the part that requires $4$ primes --- the part that concludes from irreducibility that the Galois group is either $S_n$ or $A_n$ works by reducing modulo one prime, so makes no restrictions on $L$. Here is the precise formulation:
\begin{theorem}\label{thm:galois}
  Let $f$ be as in Theorem \ref{thm:irred} but for any $L\ge 2$ (i.e.\ without the restriction that $L$ be divisible by $4$ primes). Then
  \[
  \lim_{n\to\infty}\PP(\mbox{the Galois group of $f$ is transitive and different from $A_n$ and $S_n$})=0.  
  \]
  \end{theorem}

The proofs of Theorems \ref{thm:irred} and \ref{thm:galois} are in \S\S\ref{sec:boooring}-\ref{sec:proof}. Section \ref{sec:boooring} contains well-known facts about the connection between random polynomials and random permutations; and well-known, or at least unsurprising, facts about random permutations. Experts can safely skip to \S\ref{sec:proof} which contains the core of the proof.  
In \S\ref{sec:simulations} we include some heuristics and simulations related to the question that we could not resolve: is the Galois group $A_n$ or $S_n$?

\emph{Note.} After this paper was put on the arXiv, Breuillard and Varj\'u announced results which hold for more general distributions of the coefficients, but are dependent upon a generalised Riemann hypothesis. See \cite{BV}.

\subsection*{Acknowledgments}
The authors wish to thank Nir Avni for a discussion that turned out to be crucial for the proof of Theorem \ref{thm:irred}. Noga Alon advised us on random permutations, saving us much time. Igor Rivin shared with us many exciting simulation results. Alexei Entin, Zeev Rudnick, Ofer Zeitouni and Shoni Gilboa helped with insight and advice.

LBS was partially supported by Israel Science Foundation grant 953/14. GK was partially supported by Israel Science Foundation grant 1369/15, by the Jesselon Foundation and by Paul and Tina Gardner.

\section{Generalities}\label{sec:boooring}
\subsection{Properties of random permutations}
\begin{lemma}\label{lem:no double divisor}
With probability tending to 1, there is no $l>\log^3n$ which divides the lengths of two distinct cycles of a random permutation.
\end{lemma}
\begin{proof}For any $k_1$ and $k_2$, the probability that both are lengths of cycles is bounded above by $1/k_2k_2$ (one may do this easy calculation oneself, or may consult the beginning of the proof of \cite[Claim 1]{LP93}). Summing over $k_1$ and $k_2$ in $l\ZZ\cap[1,n]$ gives that the probability that there are two cycles lengths divisible by $l$ is bounded above by $C(\log n)^2/l^2$. Summing over $l\ge \log^3 n$ gives the lemma.
\end{proof}
%% The next lemma is well known, and we include its proof for completeness.
%% \begin{lemma}\label{lem:process}
%%   The lengths of the cycles of a random permutation on $n$ letters can be constructed by the following process. Choose a uniform element $x_1$ in $\{0,\dotsc,n-1\}$, and let $n-x_1$ be the first cycle length. Next choose a uniform element $x_2$ in $\{0,\dotsc,x_1-1\}$ and let $x_1-x_2$ be the second cycle. Continue until some $x_i$ is $0$.
%% \end{lemma}
%% \begin{proof}Choose some element of $\{1,\dotsc,n\}$, say 1, and examine its cycle. It is easy to see that the probability that the cycle of 1 has size $k$ is exactly $\frac 1n$, since there are $\binom{n-1}{k-1}$ ways to choose the cycle elements, $(k-1)!$ ways to choose the cycle from the chosen elements, and $(n-k)!$ ways to choose the remaining items. Take the cycle of 1 to be $n-x_1$. The remaining permutation is random on $x_1$ letters, so continue inductively.
%% \end{proof}
\begin{lemma}\label{lem:unique prime}
For any $0\le a<b\le 1$, with probability tending to 1, a random permutation has a cycle whose length, $l$, satisfies the following two requirements:
\begin{itemize}
\item $l\in[n^a,n^b]$.
\item $l$ has a prime factor $p$ such that $p>\log^3 n$.
\end{itemize}
\end{lemma}
\begin{proof}
By \cite[Theorem 2.9, ``Cycles in sets theorem'']{F}, for any set $T\subset\{1,\dotsc,n\}$, the probability that none of the cycles is in $T$ is bounded by $\exp(-H+1)$ where $H=H(T):=\sum_{k\in T}\frac 1k$ (for easier comparison to \cite{F}, we take the $r$ there to be 1, and the $k_1$ there to be 0). So we need only calculate $H$.

The probability that a random number $k$, uniform between 1 and $n^b$, has all its prime divisors smaller than $\log^3 n$ has a well-known estimate: it is equal to $n^{-b/3+o(1)}$, (see \cite[Equation~7.16, page 203]{DM}, and the definition of their $\psi$ on the top of page 202). Hence we may estimate
\[
H\ge \sum _{\mathclap{k=n^a+n^{2b/3+o(1)}}}^{n^b} 1/k=(b-\max\{a,\tfrac 23b\}+o(1))\log n.
\]
As the constant is positive, the lemma is proved.
\end{proof}
\subsection{Polynomials versus Permutations}\label{sec:p vs p}
In this section we discuss the fact that the cycle structure of a random permutation is similar to the decomposition of a random polynomial to irreducible factors. In a way it goes back to Gauss (who showed that the probability that the random polynomial is irreducible is close to $\frac 1n$, which is the probability that the permutation has only one cycle), and was developed in the literature significantly, say in \cite{ABT}. Still, we need a few lemmas which we did not find in the literature.

Consider the space $\Omega$ of all tuples $(m_1,m_2,\ldots)$ of nonnegative integers with finite support; i.e., $m_i\geq 0$ for all $i$ and $m_i=0$ for all sufficiently large $i$. 
We define two sequences of random variables on  $\Omega$. First, for $n\geq 1$, let $f$ be a random monic polynomial of degree $n$ in $\FF_q[T]$; i.e., the $0,\dotsc,n-1$ coefficients of $f$ are i.i.d.\ uniform in $\FF_q$ and the $n^\textrm{th}$ coefficient is 1; and let
$X_n(m_1,m_2,\ldots )$ be the probability that $f$ has $m_i$ prime factors of degree $i$ in its prime factorization. In particular, $\sum_i i m_i=n$ if $X_n(m_1,m_2,\ldots )>0$. Similarly, let $Y_{n}(m_1,m_2,\ldots )$ be the probability that a random permutation on $n$ letters has $m_i$ cycles of length $i$ in its decomposition to a product of disjoint cycles. Again, $\sum_i im_i=n$ if $Y_n(m_1,m_2,\ldots)>0$.  
If $c_{i,m}$ denotes the number of possibilities to choose $m$ unordered monic irreducible polynomials of degree $i$ and if 
\[
\alpha(i,m) = \frac{c_{i,m}}{q^{im}},
\]
then we have the formulas
\begin{equation}\label{eq:X,Ydef}
\begin{split}
Y_n(m_1,m_2,\ldots) &= \prod_{i} \frac{1}{m_i! i^{m_i}}, \qquad\text{and}\\
X_n(m_1,m_2,\ldots) &= \prod_{i} \alpha(i,m_i).
\end{split}
\end{equation}
We denote by $\PP_{X_n}$ and $\PP_{Y_n}$ the probabilities that $X_n$ respectively $Y_n$ induce on $\Omega$.

For $m=1$ we have the exact formula $i\alpha(i,1) = \sum_{j\mid i} \mu(i/j) q^{j-i}$, with $\mu$ the M\"obius function, which implies that 
\begin{equation}\label{eq:PPT}
-\frac{2q^{-i/2}}{i}\leq \alpha(i,1)-\frac{1}{i}\leq 0,
\end{equation}
see, e.g., \cite[Lemma~4]{Pollack}. We can use \eqref{eq:PPT} to get
\begin{equation}\label{eq:Estalpha(im)}
\alpha(i,m) =\frac{1}{m!i^m}\exp\big(O(mq^{-i/2}+m^2 i q^{-i})\big).
\end{equation}
Indeed, as the number of ways to choose $m$ unordered objects out of $x$ objects with repetition is $\binom{m+x-1}{m}$, one has
\[
\alpha(i,m) %=\frac{1}{q^{im}} \binom{q^i\alpha(i,1)+m-1}{m} 
= \frac{\big(\alpha(i,1) + q^{-i}(m-1)\big)\big(\alpha(i,1) + q^{-i}(m-2)\big)\cdots \alpha(i,1)}{m!}
\]
So plugging \eqref{eq:PPT} to this equation we get 
\begin{align*}
\alpha(i,m) 
	& = \frac{1}{m! i^m} \prod_{j=1}^m \left(1 +O(q^{-i/2} + q^{-i}(m-j)i)\right) \\
	& = \frac{1}{m! i^m} \exp\big( O(mq^{-i/2} + m^2 i q^{-i}) \big),
\end{align*}
proving \eqref{eq:Estalpha(im)}.
%			&= \frac{(1/i +O(q^{-i/2}/i+m^2q^{-i}))^m}{m!} = \frac{1}{m!i^m}  + O(m(q^{-i/2}+m^{2}q^{-i})).
%\end{align*}

Normally, we will use this with $m$ small relative to $q^{i/2}$. For example, if we assume that $m\leq q^{i/4}$, and in general $i\ll q^{i/4}$, so  \eqref{eq:Estalpha(im)} gives 
\begin{equation}\label{eq:Estalpha(im)smallm}
\alpha(i,m) = \frac{1}{m!i^m} (1 + O(q^{-i/4})) .
\end{equation}
Two useful equalities, which hold for all $x\le n$ are:
\begin{align}
\sum_{\substack{(m_1,\dotsc,m_n)\\\sum im_i=x}}
\prod_{i=1}^n\frac{1}{m_i!i^{m_i}}&= 1\label{eq:onex}\\
\sum_{\substack{(m_1,\dotsc,m_n)\\n\ge \sum im_i\ge x}}
\prod_{i=1}^n\frac{1}{m_i!i^{m_i}}&= n-x+1\label{eq:manyx}
\end{align}
where (\ref{eq:onex}) comes from noting that the terms summed over are exactly the ones which correspond to $Y_x$, so they are probabilities and sum to 1, and (\ref{eq:manyx}) is simply the sum of (\ref{eq:onex}) from $x$ to $n$.
We will also need auxiliary lemmas, that allow us to reduce to the case with $m_i=0$ for small $i$:
\begin{lemma}\label{lem:miexponentially}
For all $n>1$, $i\in \{1,\ldots, n\}$ and $\lambda\geq 0$, 
\begin{equation}\label{fixingm_i}
\PP_{X_n}(m_i= \lambda) \leq e^{-c\lambda},
\end{equation}
where $c>0$ is a positive constant. In particular, 
\begin{equation}\label{fixingm_i1}
\PP_{X_n}(m_i\geq \lambda) \ll e^{-c\lambda}.
\end{equation}
Similar estimates hold also for $Y_n$.
\end{lemma}

\begin{proof}
By \eqref{eq:X,Ydef},
\[
\PP_{X_n}(m_i=\lambda) = \alpha(i,\lambda)
\sum_{\substack{m_i=0\\ \sum_{j\neq i} j m_j = n-\lambda} }\prod_{j\neq i} \alpha(j,m_j).
\]
The sum on the right hand side is smaller than the same sum without the restriction $m_i= 0$, which is simply $1$ (compare to (\ref{eq:onex})), so 
\begin{equation}\label{eq:malpha}
\PP_X(m_i=\lambda) \leq \alpha(i,\lambda).
\end{equation}
Thus it suffices to show that $\alpha(i,\lambda)\leq e^{-c\lambda}$. A similar argument shows that $\PP_{Y_n}(m_i=\lambda)\le 1/\lambda!i^\lambda$ which finishes the $Y_n$ case and we will not return to it.

For $i>1$, \eqref{eq:PPT} gives
\[
\alpha(i,1) \leq \frac{1}{i}\leq \frac12. 
\]
Since $\alpha(i,\lambda)\leq \alpha(i,1)^\lambda$, we get the needed bound
\[
\alpha(i,\lambda) \leq \frac{1}{2^{\lambda}}= e^{-\lambda\log 2}.
\]

For $i=1$ and $\lambda=2$ there are $\binom{q}{2}+q$ ways to choose  two linear polynomials, hence, as $q\geq 2$, 
\[
\alpha(1,2) = \frac{\binom{q}{2}+q }{q^2} = \frac{1}{2}(1+1/q) \leq \frac{3}{4}.
\]
This also does the case $i=1$ and $\lambda>2$ since
\[
\alpha(1,2\lambda+1)\leq \alpha(1,2\lambda) \leq \alpha(1,2)^{\lambda}\leq e^{-\lambda(\log4-\log 3)}.
\]
The last remaining case is $i=\lambda=1$, for which we forgo (\ref{eq:malpha}) and estimate $\PP(m_1=1)$ directly (we just need an estimate uniform in $n\ge 2$ and $q$). For any linear polynomial $p$ we have
\[
\PP(p\;|\;f)=q^{-1}
\]
since there are exactly $q^{n-1}$ monic polynomials of degree $n-1$, each one may be multiplied by $p$ to get a monic polynomial of degree $n$, and these are all different. Similarly, if $p_1$, $p_2$ and $p_3$ are linear polynomials we have
\[
\PP(p_1p_2\;|\;f)=q^{-2}\qquad \mbox{and}\qquad\PP(p_1p_2p_3\;|\;f)\le q^{-3},
\]
where the inequality in the second case is simply because we only assumed $n\ge 2$ and if $n=2$ this probability is 0. Using inclusion-exclusion gives %***I found this on the internet as ``Bonferroni inequalities'' but for me this is just the inclusion-exclusion principle, so I suggest we don't use the name Bonferroni***
\begin{equation*}
  \PP(\exists p\mbox{ linear such that }p\;|\;f)
  \le q\cdot q^{-1}-\tbinom{q}{2}q^{-2}+\tbinom{q}{3}q^{-3}=\frac{2}{3}+\frac{1}{3q^2}\le \frac{3}{4}.
\end{equation*}
This was the last remaining case so the proof of \eqref{fixingm_i} is done.

Summing over all integers $\geq \lambda$ gives
\[
\PP_{X_n}(m_i\geq \lambda) = \sum_{\mu=\lambda}^{\infty} \PP_{X_n}(m_i= \mu) \leq \sum_{\mu=\lambda}^{\infty} e^{-c\mu} =e^{-c\lambda} \frac{1}{1-e^{-c}},
\]
which proves  \eqref{fixingm_i1}.
\end{proof}

In the following lemma we write $d_{TV}(A,B)$ to denote the total variation distance between $A$ and $B$.

\begin{lemma}\label{lem:ABT}Let $X_{n,r}$ be the measure on vectors $(m_r,m_{r+1},\dotsc)$ given by restricting $X_n$, i.e.\ $X_{n,r}(m)$ is the probability that a random polynomial of degree $n$ has $m_r$ factors of degree $r$, $m_{r+1}$ factors of degree $r+1$ etc. Let $Y_{n,r}$ be the analogous quantity for $Y_n$, the measure on cycles of random permutations. Then
  \[
  d_{TV}(X_{n,r},Y_{n,r})\le C/r.
  \]
\end{lemma}
\begin{proof}This is Theorem~5.8 in \cite{ABT} --- to aid the reader in understanding the notation of \cite{ABT}, their $Y_j$ is our $m_j$ for permutations, their $C_j$ is our $m_j$ for polynomials (both are defined on page 349 of \cite{ABT}) and their notation $\mathcal{L}$ is the standard notation for ``the law of a random variable''. Let us note that $C/r$ is suboptimal --- one may show an exponential decay in $r$ --- but we will not need this extra precision.
  \end{proof}

\section{4 independent permutations}\label{sec:proof}
\begin{lemma}\label{lem:Konyagin}
There exists an $\omega:\NN\to\NN$ with $\lim_{n\to\infty}\omega(n)=\infty$ such that 
\[
\lim_{n\to\infty}\PP(f\textrm{ has a divisor of degree }\le \omega(n))=0
\]
where $f$ is as in Theorem~\ref{thm:irred}.
\end{lemma}
(this lemma does not require $L$ to be divisible by 4 distinct primes)
\begin{proof}This is well known and has many proofs in the literature. By far the best $\omega$ was achieved by Konyagin \cite{K99} who showed this with $\omega(n)=n/\log n$. The statement in \cite{K99} is only for coefficients 0 and 1, but the proof carries through in our case. A simpler argument that gives only $\omega(n)=\sqrt{\log n}$ can be found in \cite[Theorem~1.10 and \S2.2]{ORW}. An even simpler argument, with no explicit bound on $\omega$,  can be found in \cite[\S 2]{KZ13}. What we tell you three times is true.
\end{proof}
\begin{lemma}\label{lem:PPR}
Let $\sigma_{1},\sigma_{2},\sigma_{3},\sigma_{4}$ be 4 independent
uniform permutations in $S_{n}$. For $i\in \{1,\dotsc,4\}$ and $l\le n$
we define $E(i,l)$ to be the event that $l$ can be written as a
sum of lengths of cycles of $\sigma_{i}$. Then for all $k<n$, 
\[
\mathbb{P}\left(\bigcup_{l=k}^{2k}\bigcap_{i=1}^{4}E(i,l)\right)\le Ck^{-c}
\]
where both constants are absolute, in particular independent of $n$
and $k$. 

Further, for an additional parameter $\lambda$,
\begin{equation}\label{eq:with lambda}
\mathbb{P}\bigg(\bigcup_{l=k}^{2k}\bigcup_{\lambda_{1}=0}^{\lambda}\dotsb\bigcup_{\lambda_{4}=0}^{\lambda}\bigcap_{i=1}^{4}E(i,l-\lambda_{i})\bigg)\le C(\lambda+1)^{4}k^{-c}
\end{equation}
\end{lemma}
\begin{proof}
  This lemma is essentially in \cite{PPR}, but not stated as such explicitly. \cite{EFG} will serve as a convenient reference (in fact, the first part is proved there explicitly, see Proposition 2.1). We may assume without loss of generality that $k$ is sufficiently large and that $\lambda < k/2$ since otherwise, with an appropriate choice of constants, the right-hand side of (\ref{eq:with lambda}) will be larger than 1. Let $\eps\in (0,\frac12)$ be some parameter. Let $B(i,k,\eps)$ be the event that $\sigma_i$ has at least $(1+\eps)\log k$ cycles whose sizes are less than $k$. The lemma is an easy corollary from the following estimates:
\begin{enumerate}
  \item $\PP(B(i,k,\eps))\le C_1(\eps)k^{-\eps^2/3}$.
  \item $\PP(E(i,k)\setminus B(i,k,\eps))\le C_2(\eps)k^{\log 2-1+2\epsilon}$.
\end{enumerate}
The first estimate follows from \cite[Lemma 2.2]{EFG} and a little bit of calculus (if you prefer to see the calculus done explicitly, see the beginning of the proof of \cite[Proposition 2.1]{EFG}, but not in the arXiv version, only in the official journal version). The second estimate is \cite[Lemma 2.3]{EFG}.

Our lemma now follows by noting that for any $l\in [k/2,2k]$ and any $k>C_3(\eps)$, $B(i,l,\eps)$ implies $B(i,2k,\epsilon/2)$. Hence
\begin{multline*}
 P\coloneqq\PP\bigg(\bigcup_{l=k}^{2k}\bigcup_{\lambda_1=0}^\lambda\dotsb\bigcup_{\lambda_4=0}^\lambda\bigcap_{i=1}^4E(i,l-\lambda_i)\bigg) \le\\
  \le \sum_{i=1}^4\PP(B(i,2k,\epsilon/2))+
  \sum_{l=k}^{2k}\sum_{\lambda_1=0}^\lambda\dotsb\sum_{\lambda_4=0}^\lambda
    \prod_{i=1}^4\PP(E(i,l-\lambda_i)\setminus B(i,l-\lambda_i,\eps)).
\end{multline*}
Since $4(\log 2-1)<-1$ we may choose $\eps$ such that $4(\log 2-1+2\eps)<-1$ ($\eps=0.01$ does the trick) and continue the calculation to get
$$
  P\le 4\cdot C_1(\eps/2)k^{-\eps^2/12}
    +(\lambda+1)^4\sum_{l=k/2}^{2k}C_3(\eps)\cdot k^{4(\log 2-1+2\eps)}
  \le C_4(\eps)(\lambda+1)^4k^{-c},
$$
as needed.
\end{proof}

%\begin{theorem}
%Let $N$ be a positive integer divisible by at least $4$ distinct primes. Let $f= \sum_{i=0}^n \zeta_i X^i$ be a polynomial, where $\zeta_1,\zeta_2,\ldots$ are i.i.d.\ random variables taking values uniformly in $\{1,\ldots, N\}$. Then 
%\[
%\lim_{n\to \infty}\PP(f \mbox{ is irreducible}) =1.
%\]
%\end{theorem}

\begin{proof}[Proof of Theorem \ref{thm:irred}]
Let us start the proof with the following reduction: it suffices to show, for every $k<n$, that the probability that $f$ has a divisor of degree between $k$ and $2k$ is smaller than $C/\log^2k$. Indeed, once this is proved, one may handle divisors of small degree using Lemma \ref{lem:Konyagin}, and then sum over $k$ running through powers of 2 from $\omega(n)$ ($\omega$ from Lemma \ref{lem:Konyagin}) to $n$. Let us, therefore, fix one $k<n$ until the end of the proof.

Let $\red_p(f)$ be the polynomial we get by reducing the coefficients of $f$ modulo $p$. Then $\red_{p}(f)$ is a random uniform polynomial in $\FF_p$, for every $p\mid L$, and the different $\red_p(f)$ are independent. For $r\in\{1,2,3,4\}$ Let $X_r$ be an $\Omega$-valued random variable which takes the value $(m_{1,r},m_{2,r},\dotsc)$ if the reduction of $f$ modulo the $r^{\textrm{th}}$ prime has $m_{i,r}$ irreducible factors of degree $i$ for all $i$. Let $\mQ$ be the event that for some $k\le l< 2k$ we may write $l=\sum il_{i,r}$ for some $l_{i,r}\le m_{i,r}$ for all $r=1,2,3,4$. Further, let $\mB$ be the event that for some $r\in\{1,2,3,4\}$ and some $i<\log^2k$ we have $m_{i,r}>\log^2k$.

Now, by Lemma \ref{lem:miexponentially}, $\PP(\mB)\le 4\cdot \log^2 k\cdot C e^{-c\log^2 k}$ which is negligible. As for $\mQ\setminus\mB$, it is contained in the event that some $k\le l<2k$ and some $\lambda_r<\log^6 k$ we may write 
\[
l-\lambda_r=\sum_{i>\log^2k}il_{i,r}\qquad l_{i,r}\le m_{i,r}
\]
for all $r\in\{1,2,3,4\}$. Denote this event by $\mR$. Then $\mR$ is invariant to changes in the first $\log^2 k$ values of $m$ and hence by Lemma \ref{lem:ABT}%better} 
\[
|\PP_{X_n}(\mR)-\PP_{Y_n}(\mR)|\le C/\log^2 k.%q^{-c\log^2 k}
\] 
%which is, again, negligible when compared to our goal estimate, $Ck^{-c}$
(Formally, Lemma \ref{lem:ABT} is formulated for a single $m$ and here we have 4, but this is equivalent. The easiest way to see this is probably to use Lemma \ref{lem:ABT} to construct a coupling between a single polynomial and a single permutation that succeeds with probability $C/\log^2k$ and then simply couple the 4 polynomials to 4 permutations independently.) Finally, $\PP_{Y_n}(\mR)$ can be estimated by Lemma \ref{lem:PPR} to get
\[
\PP_{Y_n}(\mR)\le Ck^{-c}\log^{24}k.
\]
We conclude that $\PP_{X_n}(\mR)\le C/\log^2k$, hence that $\PP_{X_n}(\mQ\setminus\mB)\le C/\log^2 k$, and hence that $\PP_{X_n}(\mQ)\le C/\log^2 k$. As explained in the first paragraph, this completes the proof of the theorem.
\end{proof}

\subsection{The Galois group}\label{sec:galois}
For a permutation $\sigma$ and an integer $k$ let $\Y(\sigma,k)$
be all permutations one may get by changing $\sigma$ in elements
belonging to cycles of $\sigma$ each of whose length does not exceed
$k$.
\begin{lemma}
\label{lem:all transitive}For any $\alpha<1-\frac{1+\log\log2}{\log2}$
the following holds. Let $\sigma$ be a random permutation. Then the
probability that there exists a transitive subgroup $G\not\ge A_{n}$ in $S_n$
such that $G\cap \Y(\sigma,n^{\alpha})\ne\emptyset$ goes to $0$ as
$n\to\infty.$
\end{lemma}
We follow \L uczak and Pyber \cite{LP93} closely (they proved that $\mathbb{P}(\exists G:\sigma\in G)\to0$ i.e.\ the same result but without allowing for small perturbations. See also \cite{EFG0} for a lower bound on the probability).
\begin{lemma}
\label{lem:primitive}Let $P$ be the probability that there exists
$G\not\ge A_{n}$ primitive such that $G\cap \Y(\sigma,n^{\alpha}) \neq \emptyset$.
Then if $\alpha<0.49$ then $P\to 0$ as $n\to\infty$.
\end{lemma}
(the rate of decay may depend on $\alpha$).
\begin{proof}
We follow \cite{LP93}. The proof of \cite{LP93} revolves around
the notion of the  \emph{minimal degree} of a permutation group. Let us
define it even though it is classical. For a permutation $\sigma$
define 
\[
\deg\sigma=\#\{i\in\{1,\dotsc,n\}:\sigma(i)\ne i\}
\]
i.e.\ the number of elements moved by $\sigma$; and for a group
$G$ of permutations we define its minimal degree by 
\[
\min\deg G=\min_{g\in G\smallsetminus\{1\}}\deg g.
\]
Then
\begin{claim}
\label{claim:mindeg}If $G\not\ge A_{n}$ is primitive then $\min\deg G\ge(\sqrt{n}-1)/2$. 
\end{claim}
\begin{proof}
There are two cases to consider. The first is that $G$ is doubly
transitive, i.e.\ for any $a\ne b$ and $c\ne d$ in $\{1,\dotsc,n\}$
one may find a permutation $\sigma\in G$ such that $\sigma(a)=c$
and $\sigma(b)=d$. This case goes back to \cite{B1892} who showed
that in this case $\min\deg G\ge n/4$, which is bigger than the required
$(\sqrt{n}-1)/2$. The other case is more recent,
having been done in \cite{B81}: Theorem~0.3 of \cite{B81} states
that for a primitive non-doubly-transitive permutation group $G$ and 
for any $a\ne b$ in $\{1,\dotsc,n\}$ there are at least $(\sqrt{n}-1)/2$
different values of $c\in\{1,\dotsc,n\}$ such that $a$ and $b$
are in different orbits of the stabiliser of $c$ (the stabiliser of $c$ is the subgroup $H=\{g\in G: g(c)=c\}$, and an orbit of $H$ is a set $A\subset\{1,\dotsc,n\}$ such that $\forall a\in A$ and $\forall h\in H$ we have $h(a)\in A$).
Let therefore $g\in G\smallsetminus\{1\}$ and $a\ne b$ be in some
non-trivial cycle of $g$. Then clearly $g$ may not be in the stabiliser
of any of the $c$ given from \cite[Theorem 0.3]{B81}, so $\deg g\ge(\sqrt{n}-1)/2$.
\end{proof}
Returning to the proof of Lemma \ref{lem:primitive}, we apply Lemma \ref{lem:unique prime} to find some cycle of our random permutation $\sigma$ whose length $l$ is in $[n^\alpha,n^{0.49}]$ and which has a prime divisor $p$ such that $p>\log^3n$ --- Lemma \ref{lem:unique prime} shows that this can be done with probability tending to 1. We apply Lemma \ref{lem:no double divisor} to see that $p$ does not divide the length of any other cycle of $\sigma$, again with probability tending to 1.

Let now $\rho$ be any permutation in $\Y(\sigma,n^\alpha)$. Because $l>n^\alpha$, $\rho$ will preserve the cycle of length $l$ from $\sigma$. With probability tending to 1, $\sigma$ has no more than $2\log n$ cycles (see, e.g., \cite[Claim 1(i)]{LP93}). Hence $\rho$ is different from $\sigma$ in no more than $(2\log n)n^\alpha$ places, and in particular the total length of all cycles of $\rho$ whose length is divisible by $p$ is no more than $l+2n^\alpha\log n$.

Let therefore $M$ be the product of all primes powers dividing lengths of cycles of $\rho$, different from powers of $p$. Then the points not fixed by $\rho^M$ are exactly points which belong to cycles of $\rho$ divisible by $p$, and by the previous discussion there are no more than $l+2n^\alpha\log n$ of those (but at least $l$). In other words, $\deg \rho^M\le Cn^{0.49}$ and $\rho^M\ne 1$. Claim~\ref{claim:mindeg} then implies (for $n$ sufficiently large) that $\rho$ cannot belong to any primitive $G\not\ge A_n$, finishing the proof.
\end{proof}
\begin{lemma}
\label{lem:imprimitive}Fix $\alpha<\delta=1-\frac{1+\log\log2}{\log2}$ and let $P$ be the probability that there exists a transitive imprimitive group $G\leq S_n$ such that $G\cap \Y(\sigma,n^{\alpha})\ne\emptyset$. Then $P\to0$
as $n\to\infty$.
\end{lemma}
\begin{proof}
If $G$ is transitive and imprimitive then there exists a nontrivial block system
preserved by $G$, i.e.\ one may write $n=rs$ with $r,s>1$ such that there is a division
of $\{1,\dotsc,n\}$ into disjoint sets $A_{1},\dotsc,A_{r}$ of common
size $s$ such that for every $\rho\in G$ and for any $i\in\{1,\dotsc,r$\},
$\rho(A_{i})=A_{j}$ for some $j$. The proof revolves around the interaction between this block system and cycles of $\rho$. Let $L=\{\ell_1,\ell_2,\dotsc\}$ be some cycle of a $\rho\in G$. Then there must exist some $i_1,\dotsc,i_k$ such that $\ell_1\in A_{i_1}, \ell_2\in A_{i_2},\dotsc,\ell_k\in A_{i_k},\ell_{k+1}\in A_{i_1}$ and then the cycle of $A_{i_j}$ repeats. In particular, $k$ must divide $|L|$, the length of $L$. And of course, $k\le r$ and $|L|/k\le s$.

Following \cite{LP93} we divide
the proof to three cases according to the value of $r$.

\subsubsection*{Case 1} $2\le r\le\exp(\log\log n\sqrt{\log n})$. Let $E_1$ be the event
that the random permutation $\sigma$ has, for each such $r$, a cycle
$L_{r}$ whose length is $>n^{0.99}$ and is not divisible by $r$.
By \cite[Claim 2]{LP93} $\mathbb{P}(E_1)\to1$ as $n\to\infty$. Let $E_2$ be the event that $\sigma$ has no more than $2\log n$ cycles. By \cite[Claim 1(I)]{LP93} $\PP(E_2)\to 1$ as $n\to\infty$. Let $E=E_1\cap E_2$. The
cycle $L_{r}$ cannot be changed by changing short cycles of $\sigma$
so it survives in any $\rho\in \Y(\sigma,n^{\alpha})$. Let $A_{i_1},\dotsc,A_{i_k}$ be the
set of $A_{i}$ which intersect $L_{r}$. We cannot have $k=r$ because $r$ does not divide the $|L_{r}|$. Since
$\bigcup A_{i_j}$ is invariant under
$\rho$ we found an invariant subset of $\rho$ of size $n(k/r)$. If we are also in $E_2$, then $\rho$ can differ from $\sigma$ in no more than $2n^\alpha\log n$ points, so the existence of an invariant subset of $\rho$ implies that $\sigma$ has an invariant subset
of size $x$ with $|x-nk/r|\le 2 n^{\alpha}\log n$. Denote this event by
$B_{k,r}$. For each $k$ and $r$, $\mathbb{P}(B_{k,r})\le n^{\alpha-\delta+o(1)}$
by Eberhard, Ford and Green \cite{EFG0}. Summing over $k$ and $r$ (a total of $n^{o(1)}$
possibilities) we see that $\mathbb{P}(\bigcup B_{k,r})\le n^{\alpha-\delta+o(1)}$.
However, under $E\smallsetminus\bigcup B_{k,r}$ no $\rho$ may preserve
any partition $A_{1},\dotsc,A_{r}$ with $2\le r\le\exp(\log\log n\sqrt{\log r})$
and this case is finished.

\subsubsection*{Case 2} $\exp(\log\log n\sqrt{\log n})\le r\le n\exp(-\log\log n\sqrt{\log n})$.
Let $E_3$ be the event that the random permutation has a cycle $L$
whose length is divisible by some prime $p>n\exp(-\log\log n\sqrt{\log n})$.
By \cite[Claim 4]{LP93} $\mathbb{P}(E_3)\to1$ as $n\to\infty$. The
cycle $L$ will appear also in any $\rho\in \Y(\sigma,n^{\alpha})$,
and of course prevents $\rho$ from preserving any partition $A_{1},\dotsc,A_{r}$
with $r$ as above, since by the above we can write $|L|=kt$ with $k\le r$ and $t\le s=n/r$, both of which are bounded by $n\exp(-\log\log n\sqrt{\log n})$. Hence this case is also finished.

\subsubsection*{Case 3} $n\exp(-\log\log n\sqrt{\log n})\le r<n$. Let $E_{1}$ be the
same event from case 1, i.e.\ the event that the random permutation
$\sigma$ has, for each $s\le\exp(\log\log n\sqrt{\log n})$, a cycle
$L_{s}$ whose length is $>n^{0.99}$ and is not divisible by $s$.
By \cite[Claim 2]{LP93}, $\mathbb{P}(E_{1})\to1$ as $n\to\infty$.
Let $E_{4}$ be the event that any two cycles $M_{1}$ and $M_{2}$
of $\sigma$ satisfy $\gcd(M_{1},M_{2})\le n^{0.9}$ (here and until
the end of the lemma we do not distinguish between cycles and their
lengths in the notation). By \cite[Claim 1 (ii)]{LP93}, $\mathbb{P}(E_{4})\to1$
as $n\to\infty$ too. Fix now some $r$ as above and let $s=n/r$.
Under $E_{1}$, there exists a cycle $L_{s}$ as above. Because $L_{s}>n^{0.99}$,
it will be preserved in any $\rho\in \Y(\sigma,n^{\alpha})$. Assume
$\rho$ preserves a partition $A_{1},\dotsc,A_{r}$ and denote again the blocks which intersect $L_s$ by $A_{i_1},\dotsc,A_{i_k}$. We cannot have $\bigcup A_{i_j}=L_{s}$ (because $s$ does not divide $L_{s}$) hence $\bigcup A_{i_j}$  must contain at least one additional cycle, denote it by $M$. But then $\gcd(M,L_{s})$ is (at least) $k$ and in particular
$$
\gcd(M,L_{s})\ge k>\frac{L_s}{s}>n^{0.99}\exp(-\log\log n\sqrt{\log n}).
$$
This means, for $n$ sufficiently large, that $M>n^{0.9}$ and hence
$M$ appears also in $\sigma$. But the appearance of both $M$ and
$L_{s}$ in $\sigma$ contradicts the event $E_{4}$. Hence we get
that the event that for some $r$ as above, some $\rho\in \Y(\sigma,n^{\alpha})$
preserves some partition $A_{1},\dotsc,A_{r}$, is contained in $(E_{1}\cup E_{4})^{c}$.
We get that the probability of this event also goes to zero with $n$.
The case is finished, and so is the lemma.
\end{proof}

\begin{proof}
[Proof of Lemma \ref{lem:all transitive}]Lemmas \ref{lem:primitive}
and \ref{lem:imprimitive} do all the work.
\end{proof}

\begin{proof}[Proof of Theorem \ref{thm:galois}]
Recall that $\Y(\sigma,k)$ denotes the set of permutations which differ from $\sigma$ only in elements which belong to cycles shorter than $k$.  By Lemma \ref{lem:all transitive}, with probability tending to 1, there is no transitive subgroup $G\not\ge A_n$ such that $G\cap \Y(\sigma,n^\alpha)\ne\emptyset$. (here $\sigma$ is a random permutation and $\alpha$ is some arbitrary number in $\big(0,1-\frac{1+\log\log 2}{\log 2}\big)$ whose exact value will play no role). Let us reformulate this in the notations of \S{} \ref{sec:p vs p}: for a random tuple $m=(m_1,\dotsc,m_n)$ let $E=E(m)$ be the event that there exists a transitive subgroup $G\not\ge A_n$ and an element $g\in G$ such that $g$ has exactly $m_i$ cycles of length $i$ for all $i\ge n^\alpha$. Then the promised reformulation is:
\[
\lim_{n\to\infty}\PP_{Y_n}(E)= 0.
\]
Further, $E$ is clearly invariant to changing cycles of $\sigma$ shorter than $n^\alpha$. Hence we may apply Lemma \ref{lem:ABT} to it. We get that
\begin{equation}\label{eq:notinsmalltransitive}
\lim_{n\to\infty}\PP_{X_n}(E)=0.
\end{equation}
(Here the underlying finite field $\FF_q$ is taken with respect to a prime $q$ which divides $L$.)

Now let $f = X^n + \sum_{i=0}^{n-1} \zeta_i X^i$, $\zeta_i\in \{1,\ldots, L\}$ be a random polynomial as in the theorem. In particular, $\bar f := \red_p(f)$ is a random monic polynomial of degree $n$ in $\FF_p[X]$. 
Let $N$ be the splitting field of $f$ over $\QQ$ in $\CC$, $R\subseteq N$ the set of roots of $f$, and $G=\Gal(N/\QQ)\leq {\rm Sym}(R)$. %By Theorem~\ref{thm:irred}, with probability tending to 1 we have that $f$ is irreducible, which is equivalent to say that $G$ is transitive.

Let $O$ be the ring of integers of $N$. %By Chevalley's Theorem \cite[Proposition 2.3.1]{FJ3} 
Take a prime ideal $\mathfrak P$ of $O$ that  lies over $p$, i.e.\ such that $\mathfrak P\cap \mathbb{Q}=p\mathbb{Z}$. Choose one such $\mathfrak P$ arbitrarily.  %, and $D=D_{\mathfrak{P}}$ be the inertia and decomposition groups; so $I\lhd D\leq G$. 
The map $O\to O/\mathfrak P$ takes $\mathbb Z$ to $\mathbb F_p$ so if we write, in $O[X]$, $f = \prod_{\rho\in R} (X-\rho)$, then we get for $\bar f$, our reduction of $f$ to $\mathbb F_p$, that $\bar f =  \prod_{\rho\in R} (X-\bar\rho)$, where $\bar \rho$ is the image of $\rho$ under the map $O \to O/\mathfrak P$. 

We may write $\bar f = \phi \psi$, with relatively prime $\phi,\psi \in \FF_p[X]$, such that $\phi$ is squarefree and $\psi$ is squarefull (i.e., the multiplicity of each irreducible factor of $\psi$ is at least $2$). 
The probability that $\bar f$ has a square of degree $k$ dividing it, is $p^{-k/2}$; hence with probability tending to $1$, $\deg \psi\leq n^{\alpha}$ (with $\alpha$ as above). 
We decompose $R$ as $R = R_\phi \cup R_\psi$, with $R_{\phi} = \{ \rho\in R : \phi(\bar\rho)=0\}$ and $R_\psi = R\smallsetminus R_{\psi}$. So the map $\rho\in R_{\phi} \mapsto \bar\rho$ surjects onto the roots of $\phi$. 
Since $\phi$ and $\psi$ are relatively prime, we cannot have $\rho\in R_\phi$ such that $\psi(\bar\rho)=0$. Thus, since $\phi$ is squarefree, the map $\rho \mapsto \bar\rho$ is a bijection from $R_\phi$ onto the roots of $\phi$. 

Now, the map $G\to\Gal((N/\mathfrak P)/\mathbb F_p)$ is onto (see, e.g., \cite[Lemma~6.1.1(a)]{FJ3}) and hence there exists an element $\tau \in G$ which maps to the Frobenius element $x\mapsto x^p$ i.e.\ satisfying
\[
\tau \rho \equiv \rho^p \mod \mathfrak P,
\]
for all $\rho \in R$. 
%Since on $R_\phi$, the map $\rho\mapsto \bar\rho$ is injective (as $\phi$ is squarefree ***Lior, actually, why is that?***, we have that
Thus $\tau$ acts on $R_\phi$ the same as the Frobenius map acts on the roots of $\phi$. The cycle lengths of the latter is the same as the degrees of the irreducible factors of $\phi$ (this is classical, and follows from the fact that Galois groups over  a finite field are generated by the Frobenius element, and the roots of each irreducible factor is an orbit of the Galois group.) The rest of the cycles of $\tau$ are of total size $\leq n^\alpha$. 
All in all, we get that the cycle lengths of $\tau$ distribute the same as of the degrees of the irreducible factors of $\bar f$ up to cycles of length $\leq n^{\alpha}$. By \eqref{eq:notinsmalltransitive}, with probability tending to $1$, the element $\tau$ can not lie in a transitive group other than $A_n$ or $S_n$. But it lies in $G$, so either $G$ is intransitive or $G=A_n$ or $G=S_n$. 
\end{proof}

\section{Heuristics and simulations about the \texorpdfstring{$A_{n}$}{alternating} vs.\ \texorpdfstring{$S_{n}$}{symmetric group} problem}\label{sec:simulations}
\begin{conjecture*}
Let $f$ be as in Theorem~1. Then
\[
\lim_{n\to\infty}\mathbb{P}(\text{the Galois group of \ensuremath{f} is \ensuremath{S_{n}}})=1.
\]
\end{conjecture*}
Let us first explain why this conjecture does not follow from our
methods. Indeed, considering reductions modulo $p$ one gets 4 elements
of the Galois group (the lifts of the Frobenius elements), whose conjugation classes are independent, and close to uniform in the sense explained above, i.e.\ with some deviations in the very smallest cycles. These deviations do not change the distribution of the sign because, even after conditioning on all the cycles of small size, the parity of the number of remaining cycles is still approximately uniform (we will not justify this fact here, but it is not difficult). Now, 4 independent uniform permutations have probability exactly $\frac{1}{16}$ to all be in $A_{n}$. The comparison techniques described in \S 2 can be used to conclude that, for a random polynomial, all 4 lifts of the Frobenius element belong to $A_n$ with probability close to $\frac{1}{16}$. Hence we get that the lower bound for the limit
of the probabilities in the conjecture is at least bounded away from $0$,
but not quite $1$.

To differentiate $S_n$ from $A_n$, one may use the discriminant. Recall the definition of the discriminant $\disc (f) = \prod_{i<j} (\alpha_i-\alpha_j)^2$, where $\alpha_1, \ldots, \alpha_n$ are the complex roots of $f$. As a symmetric expression in the roots, $\disc(f)$ is an integer. The basic Galois theoretic property of $ \disc (f)$, for separable $f$, is that $G_f \leq A_n$ if and only if $ \disc(f)$ is a perfect square. Therefore, in order to show in Theorem~\ref{thm:galois} that $\PP(G_f = S_n)\to 1$, we have to show that $\PP(\Delta(f) \neq \square)\to 1$.

Hence it makes sense to study $\disc( f)$ for a random $f$ (say in the model of $\pm 1$ coefficients). Simulations
done by Igor Rivin show that $\log|\disc (f)|$ has an asymptotically
normal law, with average and variance both linear in $n$. It would
be interesting to prove that rigorously, maybe even for the case that
$f$ has gaussian coefficients (in the gaussian case, extremely fine
estimates have been shown for the distribution of the zeroes of $f$,
we covered some references in the introduction). This gives the following crude heuristic: the discriminant is a random very large (exponential in $n$) integer, so the probability that it is a square should be very small. We performed simulations of the probability that the discriminant is a square, and it seems to decay exponentially in the degree, though there are also arithmetic effects: for example,  the discriminant of $\sum_{i=0}^n\pm x^i$ can never be a square when $n\equiv 2$ or $4$ mod $8$. Here is a sketch of an argument by Alexei Entin:

Since $f(x)  = (x^{n+1}-1)/(x-1)$ mod 2, and since $n$ is even we have that $\gcd(x^{n+1}-1,(n+1)x^n)=1$ and in particular $x^{n+1}-1$ is square-free modulo 2, hence so is $f$.
This means that the Galois group of $f$ over the field of 2-addic numbers $\QQ_2$, as a permutation group on the roots, is isomorphic to the Galois group modulo $2$, and so it is same for any choice of $f$. 
Now, it is easy to check that the discriminant of $f$ is not a square in $\QQ_2$ when $n\equiv 2,4 \mod 8$.

Another fact discovered during simulations is that the sign is approximately evenly distributed (though it seems the inhomogeneity does not decay as $n\to\infty$ and depends on arithmetic properties of $n$). As Ofer Zeitouni remarked to us, the sign of the discriminant is simply $(-1)^{\textrm{(number of non-real roots)}/2}$ since the contribution of $(\lambda_i-\lambda_j)^2$ is positive if they are both real, and if $\lambda_j$ is non-real then the contribution of $(\lambda_i-\lambda_j)^2(\lambda_i-\overline{\lambda_j})^2$ is also positive, and similarly in the case that they are both non-real. The only negative contributions come from $(\lambda_i-\overline{\lambda_i})^2$. A lot is known about the number of real zeroes, but to the best of our knowledge their parity has not been studied.

Other interesting phenomena discovered in simulations relate to the powers of 2 that may divide $\disc(\sum\pm x^i)$. Denote by $k$ the maximal number such that $2^k$ divides the discriminant. Then there are many connections between $k$ and the degree, $n$. Let us present simulations for $n\le 100$:
\begin{itemize}
\item If $n$ is even then $k=0$. As explained above, this is because modulo 2 our polynomial is always $(x^{n+1}-1)/(x-1)$ and hence is square-free, so the discriminant is odd.
\item If $n$ is odd then always $k\ge n-1$. Shoni Gilboa gave a beautiful proof of this fact, which we will only sketch: write $\disc(f)=\det(MM^*)$ where $M$ is the Vandermonde matrix $\alpha_i^j$ and $\alpha_i$ are still the roots of $f$. The entries of $MM^*$ can be related to the coefficients of $f$ using the \emph{Newton identities} and it is a simple inductive check that all entries turn out to be odd integers. Subtracting the first row from all the others one can pull out $2^{n-1}$ and still get an integral matrix.
\item The tail of the distribution of $k$ is much fatter than we expected. Values of $120$ and more are easily observed in simulations (say with $10^5$ runs). This is especially true if $n\equiv 3$ mod 4.
\item If $n\equiv 7$ mod 8 then $k$ cannot take the values $n$ and $n+2$.
\item If $n\equiv 3$ mod 8 then $k$ cannot take the values $n$ and $n+4$. Sometimes the value $n+10$ is also prohibited: in our simulations the values of $n$ for which this happened were 27, 43, 51, 67, 75, 91 (notice the absence of 99, so this is not related to modulo 24). 
\item If $n\equiv 1$ mod 4 then $k$ is much more restricted. Occasionally ($n=37, 57, 73, 81, 93$), it may only take the value $n-1$. Sometimes it is restricted to two values ($n=9, 21, 25, 33, 45, 85$). And very typically it is restricted to an arithmetic progression (starting from $n-1$), see table \ref{table:wtf} for the jumps. The jump always divides $n-1$, so the arithmetic progression can be thought of as starting from 0. Notice two values of $n$, 41 and 69, for which the jump is 1, i.e.\ all values are allowed. But these are the only exceptions that came up in our simulations.
  \end{itemize}
We are not sure how all this reflects on the probability that the discriminant be a square, but we thought it is interesting enough to mention.

\begin{table}

\begin{tabular}{|c|c|c|c|c|c|c|c|c|c|}
\hline 
degree & jump & degree & jump & degree & jump & degree & jump & degree & jump\tabularnewline
\hline 
\hline 
9 & 4 & 29 & 2 & 49 & 4 & 69 & 1 & 89 & 2\tabularnewline
\hline 
13 & 3 &  33 & 8 & 53 & 2 & 73 & & 93 & \tabularnewline
\hline 
17 & 2 &  37 &  & 57 &  & 77 & 2 & 97 & 3\tabularnewline
\hline 
21 & 10 & 41 & 1 & 61 & 5 & 81 & \tabularnewline
\hline 
25 & 12 & 45 & 11 & 65 & 2 & 85 & 14\tabularnewline
\hline 
\end{tabular}

\caption{The arithmetic progressions of allowable values of $k$. Holes indicate values of the degree where only one value of $k$ is allowed.}\label{table:wtf}

\end{table}

Finally, we studied the probability that the determinant of a random
matrix with $\pm1$ entries is a square. Since determinant and discriminant
share both the `d' in the beginning and the `minant' at the end, this
seems quite relevant (also, of course, the discriminant has at least two formulas as the determinant of a matrix whose entries are functions of the coefficients: one as the determinant of $MM^*$ that we mentioned above, and another that comes from the fact that it is the resultant of $f$ and $f'$, giving a $(2n+1)\times(2n+1)$ matrix whose entries are multiples of the coefficients of $f$). In this case we did manage to prove the following:
\begin{theorem}
\label{thm:deter}Let $M$ be an $n\times n$ matrix with i.i.d.\ entries
taking the value $0$ with probability $\frac 12$ and the values $1$ and $-1$ with probability $\frac{1}{4}$ each.
Then 
\[
\lim_{n\to\infty}\PP(\exists k\in \mathbb{Z} \textrm{ s.t.\ } \det M=k^2)=0.
\]
%where $\mathcal{S}=\{n^{2}:n\in\mathbb{Z}\}$.
\end{theorem}

The proof of Theorem~\ref{thm:deter} is too long to include here, so we put it in a complementary paper \cite{nubeemet}.

\end{document}